\documentclass{amsart}
\allowdisplaybreaks[1]

\DeclareFontEncoding{OT2}{}{}
\DeclareFontSubstitution{OT2}{cmr}{m}{n}

\DeclareFontFamily{OT2}{cmr}{\hyphenchar\font45 }
\DeclareFontShape{OT2}{cmr}{m}{n}{%
   <5><6><7><8><9>gen*wncyr%
   <10><10.95><12><14.4><17.28><20.74><24.88>wncyr10}{}

\DeclareMathAlphabet{\mathcyr}{OT2}{cmr}{m}{n}
\DeclareMathAlphabet{\mathcyb}{OT2}{cmr}{b}{n}
\SetMathAlphabet{\mathcyr}{bold}{OT2}{cmr}{b}{n}

\usepackage{amstext}
\usepackage{amsthm}
\usepackage{amssymb}

\newtheorem{thm}{Theorem}[section]
\newtheorem{lem}[thm]{Lemma}
\newtheorem{prop}[thm]{Proposition}

\theoremstyle{definition}
\newtheorem{defn}[thm]{Definition}

\theoremstyle{remark}
\newtheorem{rem}[thm]{Remark}

\newcommand{\N}{\mathbb{N}}
\newcommand{\Q}{\mathbb{Q}}
\newcommand{\R}{\mathbb{R}}
\newcommand{\Z}{\mathbb{Z}}
\newcommand{\tht}{\tilde{\theta}}

\begin{document}

\title[Quasi-derivation relations for MZVs revisited]{Quasi-derivation relations for multiple zeta values revisited}

\author{Masanobu Kaneko}
\address[Masanobu Kaneko]{Faculty of Mathematics, Kyushu University
 744, Motooka, Nishi-ku, Fukuoka, 819-0395, Japan}
\email{mkaneko@math.kyushu-u.ac.jp}

\author{Hideki Murahara}
\address[Hideki Murahara]{Nakamura Gakuen University Graduate School,
 5-7-1, Befu, Jonan-ku, Fukuoka, 814-0198, Japan} 
\email{hmurahara@nakamura-u.ac.jp}

\author{Takuya Murakami}
\address[Takuya Murakami]{Graduate School of Mathematics, Kyushu University, 
 744, Motooka, Nishi-ku, Fukuoka, 819-0395, Japan}
\email{tak\_mrkm@icloud.com}


\begin{abstract}
 We take another look at the so-called quasi-derivation relations in the theory of multiple zeta values,
 by giving a certain formula for the quasi-derivation operator. In doing so, we are not only able to 
 prove the quasi-derivation relations in a simpler manner but also give an analog of the 
quasi-derivation relations for finite multiple zeta values.
\end{abstract}

\maketitle

\section{Introduction}

The {\it quasi-derivation relations} in the theory of multiple zeta values is a generalization,
proposed by the first-named author and established by T.~Tanaka, of a set of linear relations
known as {\it derivation relations}, which we are first going to recall.

We use Hoffman's algebraic setup  (\cite{Hof97}) with a slightly different convention.
Let $\mathfrak{H}:=\mathbb{Q} \left\langle x,y \right\rangle$ be the noncommutative polynomial algebra 
in two indeterminates $x$ and $y$.  This was introduced in order to encode multiple zeta values in the 
way the monomial $yx^{k_1-1}yx^{k_2-1}\cdots yx^{k_r-1}$ corresponds to the multiple zeta value
\[ \zeta(k_1,k_2,\dots, k_r):=\sum_{0< n_1<\cdots <n_r} \frac {1}{n_1^{k_1} n_2^{k_2}\cdots n_r^{k_r}} \]
when $k_r>1$, which is a real number as the limiting value of a convergent series. 
If we denote by $Z$ the $\Q$-linear map from $y\mathfrak{H} x$ to $\R$ assigning each monomial
$yx^{k_1-1}yx^{k_2-1}\cdots yx^{k_r-1}$ to $\zeta(k_1,\dots, k_r)$, the derivation relations state that
\[ Z(\partial_n(w))=0 \]
for all $n\ge1$ and $w\in y\mathfrak{H} x$. 
Here the operator $\partial_n$ is a $\Q$-linear derivation on $\mathfrak{H}$ determined uniquely by 
$\partial_n(x)=y(x+y)^{n-1}x$ and $\partial_n(y)=-y(x+y)^{n-1}x$. Set $z=x+y$, so that $\partial_n(z)=0$.
We use this repeatedly in the sequel.

In order to introduce the quasi-derivation relations, we first define a $\mathbb{Q}$-linear map 
$\theta:=\theta^{(c)}\colon\mathfrak{\mathfrak{H}\to\mathfrak{H}}$ with a parameter $c\in\Q$
(we often drop $c$ from the notation) by setting 
\[ \theta(u)=uz=u(x+y)\  \text{ for }\  u=x,\,y \] 
and requiring
\begin{equation*}
 \theta(ww')=\theta(w)w'+w\theta(w')+cH(w)\partial_{1}(w')
\end{equation*} 
for $w,w'\in\mathfrak{H}$,
where $H$ is the $\mathbb{Q}$-linear map from $\mathfrak{H}$ to itself defined by $H(w)=\deg(w)\cdot w$ 
for any monomial $w\in\mathfrak{H}$ ($\deg(w)$ is the degree of $w$).
This is well defined because $H$ is a derivation on $\mathfrak{H}$.
Now we define the quasi-derivation map $\partial_{n}^{(c)}$. Write $\mathrm{ad}(\theta)$ the adjoint operator
by $\theta$, i.e.,  $\mathrm{ad}(\theta)(\partial):=[\theta, \partial]=\theta\partial -\partial\theta$. 
\begin{defn}
 For each positive integer $n$ and any rational number $c$, 
 we define a $\mathbb{Q}$-linear map $\partial_{n}^{(c)}\colon\mathfrak{H}\to\mathfrak{H}$ by
 \[
  \partial_{n}^{(c)}:=\frac{1}{(n-1)!} \mathrm{ad}(\theta)^{n-1}(\partial_1). 
 \] 
\end{defn}
Then the quasi-derivation relations of Tanaka \cite{Tan09} is stated as
\[ Z(\partial_n^{(c)}(w))=0 \]
for all $n\ge1$, $c\in\Q$, and $w\in y\mathfrak{H} x$.
Our aim in this paper is to take another look at this relation, or rather at the operator $\partial_{n}^{(c)}$.

\begin{rem}
1) We have changed the definition of $\theta=\theta^{(c)}$ by shifting the original (\cite{Kan07,Tan09}) by 
the derivation $w\to [z, w]/2=(zw-wz)/2$.  However, we can check that this does not change 
$\partial_n^{(c)}(w)$.  Note also that the convention of the order of the product in $\mathfrak{H}$ there is
opposite from ours. 

2)  As noted in \cite{IKZ06}, the special case $c=0$ gives the original derivation $\partial_n$: 
$\partial_n=\partial_{n}^{(0)}$. This together with works of 
Connes-Moscovicci \cite{CM1, CM2} motivated us to define $\partial_n^{(c)}(w)$ in~\cite{Kan07}.

3)  From $\theta(z^r)=rz^{r+1}\, (r\ge1)$ and $\partial_n(z)=0$, we see that 
$\partial_n^{(c)}(wz)=\partial_n^{(c)}(w)z$ and $\partial_n^{(c)}(zw)=z\partial_n^{(c)}(w)$.  
We need to use this at several points later.
\end{rem}

\section{Main Theorem}

We present a formula for $\partial_n^{(c)}(w)$ when $w$ is in $\mathfrak{H} x$.
To describe the formula, we define a product $\diamond$ on $\mathfrak{H}$ 
introduced in Hirose-Murahara-Onozuka \cite{HMO19} by
\begin{equation}\label{dia}
w_1\diamond w_2:=\phi\bigl(\phi(w_1)*\phi(w_2)\bigr)\quad(w_1, w_2\in\mathfrak{H}),
\end{equation}
where $\phi$ is an involutive automorphism of $\mathfrak{H}$ determined by
\[ \phi(x)=z=x+y\  \text{ and }\ \phi(y)=-y,\]
 and $*$ is the harmonic product on $\mathfrak{H}$ (see \cite{Hof97, HS} for the precise definition of $*$). 
This is an associative and commutative binary operation with $1\diamond w=w\diamond 1=w$
for any $w\in\mathfrak{H}$.
In \cite{HMO19}, the definition of $\diamond$ is given in an inductive manner like the
definition of $*$ in \cite{HS}. Later we only use the shuffle-type equality
\begin{equation}\label{diaind}
xw_1 \diamond yw_2 =x(w_1 \diamond yw_2) + y(xw_1 \diamond w_2),
\end{equation}
which holds for any $w_1,w_2\in\mathfrak{H}$.

We define a specific element $q_n=q_n^{(c)}$ in $\mathfrak{H}$ for each $n\ge1$ as follows.

\begin{defn}  Let $\tht=\tht^{(c)}$ be the map from $\mathfrak{H}$ to itself given by
\[ \tht(w):=\theta(w)+c H(w)y\ \ (w\in\mathfrak{H}). \]
 For each positive integer $n$, we define
 \[
  q_{n}:=\frac{1}{(n-1)!}\tilde{\theta}^{n-1}(y).
 \]
We thus have $q_1=y$ and $q_{n}=\tilde{\theta}(q_{n-1})/(n-1)$ for $n\ge2$.
\end{defn}
Note that $q_n=q_n^{(c)}$ is in $y\mathfrak{H}$, as can be seen inductively by the definition.
We shall give an explicit formula for $q_n$ in the next section. Here is our main theorem.

\begin{thm}\label{main}
For all $n\ge1$ and $c\in\Q$, we have  
 \[
  \partial_{n}^{(c)}(wx)=(w\diamond q_{n})x \quad (w\in\mathfrak{H}).
 \]
\end{thm}

Assuming the theorem, it is straightforward to deduce the quasi-derivation relations 
from Kawashima's relations (strictly speaking, its ``linear part''). Recall the linear part of
Kawashima's relations \cite{Kaw09} asserts that 
\[ Z(\phi(w_1*w_2)x)=0 \]
for any $w_1, w_2 \in y\mathfrak{H}$. Using this and the definition \eqref{dia} of $\diamond$, 
we see that 
\[ Z\bigl(\partial_{n}^{(c)}(ywx)\bigr)=Z\bigl((yw\diamond q_{n})x\bigr)
=Z\bigl(\phi\bigl(\phi(yw)*\phi(q_n)\bigr)x\bigr)=0\]
because both $\phi(yw)$ and $\phi(q_n)$ are in $y\mathfrak{H}$. This is the quasi-derivation relations.

Another immediate corollary to the theorem is the commutativity of the operators $\partial_{n}^{(c)}$, that
is, $\partial_{n_1}^{(c_1)}$ and $\partial_{n_2}^{(c_2)}$ commute with each other for any $n_1,n_2\ge1$
and $c_1,c_2\in\Q$. This was proved in \cite{Tan09} but the argument was quite involved.
Here we may show 
 \[
  [\partial_{n_1}^{(c_1)},\partial_{n_2}^{(c_2)}] (w)=0
 \]
first for $w\in\mathfrak{H}x$ as 
 \begin{align*}  
  [\partial_{n_1}^{(c_1)},\partial_{n_2}^{(c_2)}] (wx)
  &=(\partial_{n_1}^{(c_1)} \partial_{n_2}^{(c_2)} -\partial_{n_2}^{(c_2)} \partial_{n_1}^{(c_1)} ) (wx) \\
  &=((w\diamond q_{n_2}) \diamond q_{n_1})x 
   -((w\diamond q_{n_1}) \diamond q_{n_2})x\\
   &=0
 \end{align*}
because the product $\diamond$ is associative and commutative, and then for the general case 
by induction on the degree of $w$ by noting $\partial_{n}^{(c)}(wz)=\partial_{n}^{(c)}(w)z$ as remarked before.

\begin{proof}[Proof of Theorem \ref{main}]

We need some lemmas.  Recall $z=x+y$.
\begin{lem} \label{lem:HMO}
 For $w_1,w_2\in \mathfrak{H}$, we have 
 \begin{align*}
  zw_1\diamond w_2 =w_1\diamond zw_2 =z(w_1\diamond w_2).
 \end{align*}
\end{lem}
\begin{proof}
 This follows from $\phi(z)=x,\, \phi(x)=z$ and 
 $xw_1 * w_2 =w_1* xw_2 =x(w_1* w_2)$.
 See also \cite{HMO19}. 
\end{proof}
\begin{lem} \label{lemA} For $w\in\mathfrak{H}$, we have $\partial_{1}(w)=w\diamond y-wy$.
\end{lem}

\begin{proof}
We proceed by induction on $\deg(w)$. 
The case $\deg(w)=0$ is obvious because $\partial_{1}(1)=0$. 
Suppose $\deg(w)\ge1$.
By linearity, it is enough to prove the equation when $w$ is of the form $zw'$ and $xw'$.  
If $w=zw'$, we have, by using the induction hypothesis and Lemma~\ref{lem:HMO},
\[  \partial_{1}(w)=\partial_{1}(zw') =z\partial_{1}(w')
  =z(w'\diamond y-w'y) =zw'\diamond y-zw'y =w\diamond y-wy.\]
 When $w=xw'$, we similarly compute (using equation~\eqref{diaind})
 \begin{align*}
  \partial_{1}(w)&=\partial_{1}(xw') =yxw'+x\partial_{1}(w') =yxw'+x(w'\diamond y-w'y) \\
  &=y(xw'\diamond 1)+x(w'\diamond y)-xw'y=xw'\diamond y-xw'y \\
  &=w\diamond y-wy. \qedhere
 \end{align*}
\end{proof}

\begin{lem}\label{lemB}   For $u\in\mathbb{Q}x+\mathbb{Q}y$, we have
 \[
  \tilde{\theta}(uw)=u\bigl(\tilde{\theta}(w)+zw+c(w\diamond y)\bigr). 
 \]
\end{lem}
\begin{proof} We only need to show the equation for $u=x\text{ and }y$. 
 By the definition of $\tilde{\theta}$, we have
 \begin{align*}
  \tilde{\theta}(uw)&=\theta(uw)+cH(uw)y\\
  &=uzw+u\theta(w)+cu\partial_{1}(w)+cuwy+cuH(w)y\\
  &=u\bigl(\tilde{\theta}(w)+zw+c(\partial_{1}(w)+wy)\bigr). 
 \end{align*}
 From Lemma~\ref{lemA}, we complete the proof. 
\end{proof}
We need one more preparatory result, which may be of interest in its own right.
\begin{prop} \label{propA}
 The $\mathbb{Q}$-linear map $\tilde{\theta}$ is a derivation on $\mathfrak{H}$ with respect 
 to the product $\diamond$, i.e., the equation
 \begin{equation}\label{derdia}
  \tilde{\theta}(w_1\diamond w_2)
  =\tilde{\theta}(w_1)\diamond w_2+w_1\diamond\tilde{\theta}(w_2)
 \end{equation}
 holds for any $w_1,w_2\in\mathfrak{H}$.
\end{prop}
\begin{proof}
 We prove this by induction on $\deg(w_1)+\deg(w_2)$.
 The case $\deg(w_1)+\deg(w_2)=0$ holds trivially:
 \[
  \tilde{\theta}(1\diamond1)
  =\tilde{\theta}(1)=0=\tilde{\theta}(1)\diamond1+1\diamond\tilde{\theta}(1).
 \]
 When $\deg(w_1)+\deg(w_2)\ge1$, 
 we first prove when $w_1$ is of the form $w_1=zw_{1}'$. 
 By the definition of $\tilde{\theta}$ and Lemmas~\ref{lem:HMO} and \ref{lemB}, we have
 \begin{align*}
  \tilde{\theta}(zw'_{1}\diamond w_{2})=\tilde{\theta}(z(w'_{1}\diamond w_{2}))
  =z\bigl(\tilde{\theta}(w'_{1}\diamond w_{2})+z(w_{1}'\diamond w_{2})+c(w_{1}'\diamond w_{2}\diamond y)\bigr).
 \end{align*}
 On the other hand, we have 
 \begin{align*}
  &\tilde{\theta}(zw_{1}')\diamond w_{2}+zw_{1}'\diamond\tilde{\theta}(w_{2}) \\
  &=z\bigl(\tilde{\theta}(w_{1}')+zw_{1}'+c(w_{1}'\diamond y)\bigr)\diamond w_{2}
   +z\bigl(w_{1}'\diamond\tilde{\theta}(w_{2})\bigr)\\
  &=z\bigl(\tilde{\theta}(w_{1}')\diamond w_{2}+w_{1}'\diamond\tilde{\theta}(w_{2})
  +z(w_{1}'\diamond w_{2})+c(w_{1}'\diamond w_{2}\diamond y)\bigr).
 \end{align*}
Hence by the induction hypothesis we obtain 
\[ \tilde{\theta}(zw'_{1}\diamond w_{2})=\tilde{\theta}(zw_{1}')\diamond w_{2}
+zw_{1}'\diamond\tilde{\theta}(w_{2}).  \]

Since the binary operator $\diamond$ is commutative and bilinear, it suffices then 
to prove equation~\eqref{derdia} only in the case where $w_1=xw_1'$ and $w_2=yw_2'$.
By using equation~\eqref{diaind} and Lemma~\ref{lemB}, we have
\begin{align*}
&\tilde{\theta}(xw'_{1}\diamond yw_{2}')\\
&=\tilde{\theta}\left(x(w'_{1}\diamond yw_{2}')+
y(xw'_{1}\diamond w_{2}')\right)\\
&=x\bigl(\tilde{\theta}(w_1'\diamond yw_2')+z(w_1'\diamond yw_2')+c(w_1'\diamond yw_2'\diamond y)\bigr)\\
&\quad +y\bigl(\tilde{\theta}(xw_1'\diamond w_2')+z(xw_1'\diamond w_2')+c(xw_1'\diamond w_2'\diamond y)\bigr)
\intertext{and}
&\tilde{\theta}(xw'_{1})\diamond yw_{2}'+xw_1'\diamond \tilde{\theta}(yw'_{2})\\
&=x\bigl(\bigl(\tilde{\theta}(w'_{1})+zw_1'+c(w_1'\diamond y)\bigr)\diamond yw_2'\bigr)
+y\bigl(\tilde{\theta}(xw'_{1})\diamond w_{2}'\bigr)\\
&\quad +x\bigl(w_1'\diamond \tilde{\theta}(yw_2')\bigr)
+y\bigl(xw_1'\diamond \bigl(\tilde{\theta}(w_2')+zw_2'+c(w_2'\diamond y)\bigr)\bigr)\\
&=x\bigl(\tilde{\theta}(w'_{1})\diamond yw_2'+w_1'\diamond \tilde{\theta}(yw_2')
+z(w_1'\diamond yw_2')+c(w_1'\diamond yw_2'\diamond y)\bigr)\\
&\quad +y\bigl(\tilde{\theta}(xw'_{1})\diamond w_{2}'+xw_1'\diamond \tilde{\theta}(w_2')
+z(xw_1'\diamond w_2')+c(xw_1'\diamond w_2'\diamond y)\bigr).
\end{align*}
From these, we see by the induction hypothesis that
\[ \tilde{\theta}(xw'_{1}\diamond yw_{2}')=\tilde{\theta}(xw'_{1})\diamond yw_{2}'
+xw_1'\diamond \tilde{\theta}(yw'_{2}) \]
holds.
\end{proof}
Now we prove Theorem~\ref{main} by induction on $n$. When $n=1$, we have
 \[   \partial_{1}^{(c)}(wx)=\partial_{1}(wx) =\partial_{1}(w)x+wyx 
 =(\partial_{1}(w)+wy)x =(w\diamond y)x=(w\diamond q_1)x
 \]
 by Lemma~\ref{lemA}.
 When $n\ge2$, we have 
 \begin{align*}
  \partial_{n}^{(c)}(wx)&=\frac{1}{n-1}ad(\theta)(\partial_{n-1}^{(c)})(wx)\\
  &=\frac{1}{n-1}\left(\theta\partial_{n-1}^{(c)}(wx)-\partial_{n-1}^{(c)}\theta(wx)\right). 
 \end{align*}
 By the induction hypothesis, we have
 \begin{align*}
  \theta\partial_{n-1}^{(c)}(wx)&=\theta((w\diamond q_{n-1})x) \\
  &=\theta(w\diamond q_{n-1})x+(w\diamond q_{n-1})xz+cH(w\diamond q_{n-1})yx\\
  &=\tilde{\theta}(w\diamond q_{n-1})x+(w\diamond q_{n-1})xz
 \end{align*}
and
 \begin{align*}
  \partial_{n-1}^{(c)}\theta(wx)&=\partial_{n-1}^{(c)}\left(\theta(w)x+wxz+cH(w)yx\right)\\
  &=(\theta(w)\diamond q_{n-1})x+(w\diamond q_{n-1})xz+c(H(w)y\diamond q_{n-1})x\\
  &=(\tilde{\theta}(w)\diamond q_{n-1})x+(w\diamond q_{n-1})xz.
 \end{align*}
 We therefore obtain by Proposition \ref{propA}
 \begin{align*}
  \partial_{n}^{(c)}(wx)&=\frac{1}{n-1}\bigl(\tilde{\theta}(w\diamond q_{n-1})
  -(\tilde{\theta}(w)\diamond q_{n-1})\bigr)x 
  =\frac{1}{n-1}\bigl(w\diamond\tilde{\theta}(q_{n-1})\bigr)x \\
  &=(w\diamond q_{n})x,
 \end{align*}
which completes the proof.
\end{proof}

\section{Explicit formula for $q_{n}$} 
We now describe the element $q_n=q_n^{(c)}$ in an explicit manner. 
For any index $\boldsymbol{l}=(l_{1},\dots,l_{s})\in\mathbb{N}^{s}$,
we define $a(\boldsymbol{l})=a(l_{1},\dots,l_{s})\in\mathbb{Q}$ (or $\in\Z[c]$ if we view $c$
as a variable) inductively by $a(1):=1$ and
\begin{equation*}
 a(\boldsymbol{l}):=\sum_{i=1}^{s}\bigl(l_{i}-1-(l_{1}+\cdots+l_{i-1})c\bigr)\,a(\boldsymbol{l}^{(i)}), 
\end{equation*}
where
\begin{align*}
 \boldsymbol{l}^{(i)}=\begin{cases}
  (l_{1},\dots,l_{i-1},l_{i+1},\dots,l_{s}) & \textrm{if}\ l_{i}=1,\\
  (l_{1},\dots,l_{i-1},l_{i}-1,l_{i+1},\dots,l_{s}) & \textrm{if}\ l_{i}>1.
 \end{cases}
\end{align*}
\begin{prop}
 For $n\ge1$, we have
 \begin{align} \label{eq1}
  q_{n}
  =-\frac{1}{(n-1)!} \sum_{|\boldsymbol{l}|=n}
   a(\boldsymbol{l}) w(\boldsymbol{l}), 
 \end{align}
 where the sum runs over all indices $\boldsymbol{l}=(l_{1},\dots,l_{s})\in\N^s$
of any length $s$ and of weight $|\boldsymbol{l}|:=l_1+\cdots+l_s=n$, 
and $w(\boldsymbol{l})=\phi(yx^{l_1-1}\cdots yx^{l_s-1})=(-1)^s yz^{l_{1}-1}\cdots yz^{l_{s}-1}$.
\end{prop}
\begin{proof}
 Let $q'_n$ denote the right-hand side of \eqref{eq1}. We prove \eqref{eq1} by induction on $n$.
 When $n=1$, we easily see $q_{1}'=y$.
 
Suppose $n\ge2$.  We want to show that $q'_{n}=\tilde{\theta}(q'_{n-1})/(n-1)$. 
 Since $\theta(z^{m})=mz^{m+1}$ and $\partial_{1}(z)=0$, we have 
 \[
  \theta(yz^{k-1})=yz^{k}+(k-1)yz^{k}=kyz^k, 
 \] 
and so
 \begin{align*}
 &\theta(yz^{k_{1}-1}\cdots yz^{k_{r}-1}) \\
 &=\sum_{j=1}^{r}yz^{k_{1}-1}\cdots yz^{k_{j-1}-1}\cdot k_jyz^{k_j}\cdot yz^{k_{j+1}-1}\cdots yz^{k_{r}-1} \\
 &\quad +c\sum_{1\leq i<j\leq r}yz^{k_{1}-1}\cdots H(yz^{k_{i}-1})\cdots\partial_{1}(yz^{k_{j}-1})
 \cdots yz^{k_{r}-1} \\
 &=\sum_{j=1}^{r}k_j\, yz^{k_{1}-1}\cdots yz^{k_{j-1}-1}yz^{k_{j}}yz^{k_{j+1}-1}\cdots yz^{k_{r}-1} \\
 &\quad -c\sum_{1\leq i<j\leq r}yz^{k_{1}-1}\cdots (k_i yz^{k_{i}-1})\cdots y(z-y)z^{k_{j}-1}yz^{k_{j+1}-1}
 \cdots yz^{k_{r}-1} \\
 &=\sum_{j=1}^{r}k_j\, yz^{k_{1}-1}\cdots yz^{k_{j-1}-1}yz^{k_{j}}yz^{k_{j+1}-1}\cdots yz^{k_{r}-1} \\
 &\quad -c\sum_{j=2}^r (k_1+\cdots+k_{j-1})yz^{k_{1}-1}\cdots yz^{k_{j-1}-1} y(z-y)z^{k_{j}-1}yz^{k_{j+1}-1}
 \cdots yz^{k_{r}-1}. 
 \end{align*}
 Since $cH(yz^{k_{1}-1}\cdots yz^{k_{r}-1})y=c(k_1+\cdots+k_r)yz^{k_{1}-1}\cdots yz^{k_{r}-1}y$,
 we finally obtain for $\boldsymbol{k}=(k_1,\ldots,k_r)$
 \begin{align*}
 &\tilde{\theta}(w(\boldsymbol{k}))\\
  &=(-1)^r\tilde{\theta}(yz^{k_{1}-1}\cdots yz^{k_{r}-1})\\
  &=(-1)^r\sum_{j=1}^{r}\bigl(k_{j}-c(k_{1}+\cdots+k_{j-1})\bigr)
  yz^{k_{1}-1}\cdots yz^{k_{j-1}-1}yz^{k_{j}}yz^{k_{j+1}-1}\cdots yz^{k_{r}-1}\\
    &\quad -(-1)^{r+1}
    c\sum_{j=1}^{r}(k_{1}+\cdots+k_{j})yz^{k_{1}-1}\cdots yz^{k_{j}-1}\cdot y\cdot yz^{k_{j+1}-1}
\cdots yz^{k_{r}-1}. 
 \end{align*} 
 If we write 
 \[ \tilde{\theta}(q'_{n-1})=-\frac1{(n-2)!}\sum_{|\boldsymbol{l}|=n}a'(\boldsymbol{l})w(\boldsymbol{l}),\]
we see from this that the coefficient $a'(\boldsymbol{l})$ of $w(\boldsymbol{l})=(-1)^syz^{l_1-1}\cdots yz^{l_s-1}$
is given exactly by $a(\boldsymbol{l})$ as defined recursively.  \end{proof}


\section{Quasi-derivation relations for finite multiple zeta values}

In this section, we briefly discuss how the quasi-derivation relations look like for ``finite'' multiple zeta values.
There are two versions, denoted $\zeta_{\mathcal{A}} (k_1,\ldots ,k_r)$ and $\zeta_{\mathcal{S}} (k_1,\ldots ,k_r)$,
of ``finite'' analogues of multiple zeta values.  The former lives in the $\Q$-algebra
$\mathcal{A}:=\prod_p\mathbb{F}_p/\bigoplus_p\mathbb{F}_p$ and the latter the quotient
$\Q$-algebra of classical multiple zeta values modulo the ideal generated by $\zeta(2)$.
It is conjectured that the two versions satisfy completely the same relations, and there is a conjectural isomorphism
between two $\Q$-algebras generated by those two versions. For more on finite multiple zeta values,
see for instance \cite{Kan19}. 

Denote by
$\mathit{Z}_\mathcal{F}$ the $\mathbb{Q}$-linear map from $y\mathfrak{H}$ to either algebra
assigning the monomial $yx^{k_1-1}\cdots yx^{k_r-1}$ to $\zeta_{\mathcal{A}} (k_1,\ldots ,k_r)$ or 
$\zeta_{\mathcal{S}} (k_1,\ldots ,k_r)$.
Then the derivation relations for finite multiple zeta values established by the second-named author~\cite{Mur17}
is the relation
\begin{equation}\label{derfin}  \mathit{Z}_\mathcal{F} (\partial_{n}(w)x^{-1}) = 0 \end{equation}
that holds for all $w\in  y\mathfrak{H}x$.

As a consequence of our Theorem~\ref{main}, we have the following.

\begin{thm}[Quasi-derivation relations for finite multiple zeta values] \label{qderF}
 For all $n\ge1$ and $c\in\Q$, we have 
 \begin{align*} 
  \mathit{Z}_\mathcal{F} (\partial_{n}^{(c)}(w)x^{-1})
  =\mathit{Z}_\mathcal{F} (wx^{-1}) \mathit{Z}_\mathcal{F} (q_{n}^{(c)})\quad (w\in y\mathfrak{H}x).  
 \end{align*}
\end{thm}
\begin{proof}  This is almost immediate from Theorem~\ref{main} if one notes 
$\mathit{Z}_\mathcal{F}\circ \phi=\mathit{Z}_\mathcal{F}$ and $\mathit{Z}_\mathcal{F}$ is 
a $*$-homomorphism (for these, see~\cite{Jar14, Kan19, KZ17}).
\end{proof}

\begin{rem}
 When $c=0$, we can easily compute that $q_n^{(0)}=yz^{n-1}$. 
 Since $\mathit{Z}_\mathcal{F}(yz^{n-1})=\mathit{Z}_\mathcal{F}\bigl(\phi (yz^{n-1})\bigr)
 = -\mathit{Z}_\mathcal{F}(yx^{n-1})=-\zeta_\mathcal{F}(n)=0$ for $\mathcal{F}=\mathcal{A}$ or $S$,
we see that Theorem~\ref{qderF} generalizes the derivation relations \eqref{derfin}.
\end{rem}

\section*{Acknowledgement}
This work was supported by JSPS KAKENHI Grant Numbers JP16H06336.


\end{document}